\documentclass{amsart}
\usepackage{amssymb}
\usepackage{amsmath}
\usepackage{amsthm}
\usepackage{amsfonts}
\usepackage{relsize}
\usepackage{enumitem}
\usepackage[utf8]{inputenc}
\usepackage{tikz}
\usepackage{tikz-cd}
\usetikzlibrary{matrix,arrows}
\usepackage[a4paper, total={6in, 8in}]{geometry}
\usetikzlibrary{positioning}
\definecolor{processblue}{cmyk}{0.96,0,0,0}

\newcommand{\QQ}{\mathbb{Q}}
\newcommand{\ZZ}{\mathbb{Z}}
\newcommand{\NN}{\mathbb{N}}

\newcommand{\pf}{\mathfrak{p}}
\newcommand{\qf}{\mathfrak{q}}
\newcommand{\mf}{\mathfrak{m}}
\newcommand{\nf}{\mathfrak{n}}

\newcommand{\Xf}{\mathfrak{X}}
\newcommand{\Yf}{\mathfrak{Y}}
\newcommand{\Zf}{\mathfrak{Z}}
\newcommand{\oh}{\mathcal{O}}

\newcommand{\spec}{\text{Spec}\,}
\newcommand{\spf}{\text{Spf}\,}
\newcommand{\spa}{\text{Spa}\,}

\newcommand{\tr}{\text{tr}}

\newcommand{\dimn}{\text{dim}\,}
\DeclareMathOperator{\Hom}{Hom}
\theoremstyle{plain}
\newtheorem{thm}{Theorem}[section]
\newtheorem{lem}[thm]{Lemma}
\newtheorem{prop}[thm]{Proposition}
\newtheorem{cor}[thm]{Corollary}
\newtheorem{introthm}{Theorem}

\newtheorem{introprop}[introthm]{Proposition}
\newtheorem{introcor}[introthm]{Corollary}
\theoremstyle{definition}
\newtheorem{defn}[thm]{Definition}
\newtheorem{exmp}[thm]{Example}
\newtheorem{rem}[thm]{Remark}
\title{The Riemannian Hebbarkeitssätze for pseudorigid spaces}
\author{João N. P. Lourenço}

\begin{document}
\newpage
\baselineskip=15pt
\begin{abstract}
We prove Riemann's theorems on extensions of functions over certain mixed characteristic analytic adic spaces, first introduced by Johansson and Newton. We use these results to reprove a theorem of de Jong identifying global sections of an $\oh_K$-flat normal formal scheme, locally formally of finite type over $\oh_K$, with locally powerbounded sections over the generic fibre.
\end{abstract}
\maketitle

\tableofcontents

\section{Introduction}

In complex analysis, the Riemannian Hebbarkeitssätze\footnote{In the English literature, these results usually go by the name of Riemann's removable singularities lemma and Hartogs' theorem, whereas in the German literature these are called \textit{Erster Riemmanscher Hebbarkeitssatz} and \textit{Zweiter Riemannscher Hebbarkeitssatz}, respectively. We will adopt the latter terminology, not simply for its greater homogeneity, but also because the ordinals serve as a mnemonic for the lower bounds on the codimension of the appearing singularity loci.} concern the possibility of globally extending certain meromorphic functions on a complex manifold $X$ with singularities along an analytic subset $Z$. Namely, the first (resp. second) Hebbarkeitssatz states that a locally bounded (resp. holomorphic) function $f$ on $X\setminus Z$ has a unique extension to $X$, if $Z$ has codimension at least $1$ (resp. at least $2$). The analogous statements hold true for normal rigid spaces over a complete nonarchimedean field, by the work of Bartenwerfer (cf. \cite{Bar} for a proof of the first Hebbarkeitssatz) and Lütkebohmert (cf. \cite[Satz 1.6]{Luet} for a proof of the second Hebbarkeitssatz and the first Hebbarkeitssatz in the case of absolute normality).

If one works instead with adic spaces, one is then able to produce interesting examples of mixed characteristic analytic adic spaces which do not live over any field. One such example is provided by the class of pseudorigid spaces over a complete discrete valuation ring $\oh_K$ of mixed characteristic. These spaces are locally isomorphic over $\spa\oh_K=\spa(\oh_K, \oh_K)$ to analytic open subsets of a formal scheme $\Xf$ fft over $\oh_K$, here regarded as an adic space (for a slightly different formulation, cf. Definition \ref{definicao dos espacos pseudo-rigidos}). Although the author was already working with this category during the fall of 2016, it was first introduced in the literature by Johansson and Newton in \cite{JN17}, to whom we essentially owe the designation of these spaces and calling our attention to some results in \cite{Abb}. We can now state our main result:

\begin{introthm}
Let $X$ be a normal pseudorigid space over $\oh_K$ and $Z$ be a Zariski closed subset of $X$. The following statements hold:
\begin{description}[style=unboxed, leftmargin=0.2cm]
\item[First Hebbarkeitssatz] If $\emph{codim}_X Z \geq 1$, then the restriction map $\oh_X^+(X) \rightarrow \oh_X^+(X\setminus Z)$ is an isomorphism of rings, i.e. every locally powerbounded function $f$ defined on the complement of $Z$ admits a unique locally powerbounded extension to $X$.
\item[Second Hebbarkeitssatz] If $\emph{codim}_X Z \geq 2$, then the restriction map $\oh_X(X) \rightarrow \oh_X(X\setminus Z)$ is an isomorphism of rings, i.e. every function $f$ defined on the complement of $Z$ admits a unique extension to $X$.
\end{description}
\end{introthm}

We should remark that some notions in the above statement such as normality, codimension and Zariski closed subsets are not defined for general adic spaces. We will define all these notions for sufficiently well behaved analytic adic spaces, called Jacobson adic spaces, in the course of Section 2.

Let us briefly describe our proof of the Hebbarkeitssätze. Due to the rigid case, one is reduced to working in the case where $X=\spa A$ is the adic spectrum of an $\oh_K$-flat normal pseudoaffinoid algebra (cf. Definition \ref{definicao das algebras pseudo-afinoides}) and $Z$ is contained in the special fibre. The general strategy follows then Lütkebohmert's article. As a first step, we verify the statements for elementary pseudoaffinoid algebras $D_n \langle X_1, \dots, X_r \rangle$, where $D_n=\oh_K[[T]] \langle \pi/T^n \rangle[1/T]$, via direct computations with norms. The second ingredient is an adaptation of Noether's normalisation lemma:

\begin{introprop} 
Let $A$ be an $\oh_K$-flat pseudoaffinoid algebra, such that $\pi \notin A^{\times}$, and let $T$ be a topologically nilpotent unit of $A$. If $\phi: k((T)) \langle X_1, \dots, X_r \rangle \rightarrow  A/\pi A$ is a finite injection of $k((T))$-affinoid algebras, then it can be lifted to a finite injection $\Phi:D_n \langle X_1, \dots, X_r \rangle \rightarrow A \langle \pi/T^n \rangle$ of pseudoaffinoid algebras, for a sufficiently large natural number $n$.
\end{introprop}

The remaining tool consists of holomorphic, resp. bounded transfer lemmata (cf. Lemma \ref{provar cenas en passant finito} and Lemma \ref{provar cenas en passant finito 1}) which work as analogues of \cite[Hilfssatz 1.7]{Luet}. For the second Hebbarkeitssatz, the aforementioned steps immediately yield the result. As for the first Hebbarkeitssatz, in virtue of the slight weaknesses of Lemma \ref{provar cenas en passant finito 1}, we are forced to develop additional arguments, such as finding Noether normalisations which are unramified at certain given regular primes. This was not the route taken by Lütkebohmert in \cite{Luet}, as he uses a result of Kiehl about rigid spaces for which we could not find a pseudorigid analogue. 

Next, we provide some applications of these results to formal schemes. For sufficiently well behaved formal schemes, one can show that a locally powerbounded function on the set of analytic points extends to the total space.

\begin{introprop}
Let $\Xf$ be a normal excellent formal scheme, which is nowhere discrete. Then, the restriction map $\oh_{\Xf}(\Xf) \rightarrow \oh_{\Xf_a}^+(\Xf_a)$ is an isomorphism of topological rings.
\end{introprop}

In the above proposition, we regard a formal scheme as an adic space and denote by $\Xf_a$ its set of analytic points. The claim is clear for principal ideals of definition and can be reduced to that case combining Zariski's Main Theorem with the theorem on formal functions. As a corollary of this and the first Hebbarkeitssatz, we obtain a new proof of a theorem of de Jong (cf. \cite[Theorem 7.4.1]{deJong}). Showing this exact statement was actually our original motivation, before discovering it had already been proven by de Jong. One also has the following corollary:

\begin{introcor}
(Scholze) Let $\oh_K$ be a complete discrete valuation ring with an algebraically closed residue field $k$. Let $\mathcal{C}$ be the category of $\oh_K$-flat normal formal schemes, locally formally of finite type over $\oh_K$. Let $\mathcal{D}$ be the category of triples $(X,Y,p)$, where $X$ is a $K$-rigid space, $Y$ is a perfect $k$-scheme and $p: \lvert X \rvert\rightarrow \lvert Y\rvert$ is a map of topological spaces (here, $\lvert X \rvert $ means just the classical rigid points). Then the functor $F:\mathcal{C}\rightarrow \mathcal{D}$ given by $\Xf \mapsto (\Xf_{\eta}, (\Xf_{\emph{red}})^{\emph{perf}}, \text{sp}_{\Xf})$ is fully faithful. 
\end{introcor}

In the above corollary $\Xf_{\eta}$ denotes the generic fibre of $\Xf$, $(\Xf_{\text{red}})^{\text{perf}}$ the perfection of its reduction and $\text{sp}_{\Xf}$ the classical specialisation mapping. This is relevant to the study of integral models of local Shimura varieties, as one may then recover it just by knowing the perfection of its special fibre and the specialisation mapping. The reason why we consider the perfection here is that, if $\breve{\mathcal{M}}$ denotes a formal scheme representing a moduli problem of $p$-divisible groups as in \cite[Chapter 3]{RZ}, then the perfection of its reduction is an affine Deligne-Lusztig variety, by \cite[Proposition 3.11]{Zhu}. The ideas behind this paragraph were explained in more detail in Scholze's talk at the Super Arbeitsgemeinschaft held in Bonn in honor of Michael Rapoport.

\subsection*{Outline}
\addtocontents{toc}{\protect\setcounter{tocdepth}{1}}
In Section 2, we review some facts on coherent sheaves over locally noetherian analytic adic spaces and apply these to define Zariski closed subsets globally in terms of coherent sheaves of ideals. In Section 3, we introduce the class of Jacobson adic spaces, whose behaviour resembles that of rigid spaces. We give a sufficient criterion based on work of \cite{Abb} to determine whether an analytic adic space is Jacobson. Then we define absolute properties such as reducedness, normality and regularity for Jacobson adic spaces, and also define the codimension of a Zariski closed subset. In Section 4, we introduce pseudorigid spaces over a complete DVR and establish some of their properties, including our version of the Noether normalisation lemma. In Section 5, we give our proofs of the Hebbarkeitssätze and the final section is concerned with the applications to formal schemes alluded to above.

\subsection*{Notation}
We fix the following notation for the rest of the paper: $K$ denotes a discretely valued complete nonarchimedean field, $\oh_K$ its ring of integers, $\pi_K$ a uniformiser of $\oh_K$ and $k$ the residue field of $\oh_K$. We will always use the letters $A$ and $B$ (resp. $X$, $Y$ and $Z$) to denote strongly noetherian complete Tate rings (resp. locally noetherian analytic adic spaces).
\subsection*{Acknowledgements}
I want to heartily thank my advisor Peter Scholze for his guidance and enthusiasm during this work, and for inviting me to give two talks on the subject at the Arbeitsgemeinschaft Arithmetische Geometrie in Bonn. Evidently, my gratitude extends to everyone present at the talks for their attention and pertinent questions. I am also indebted to Mafalda Santos, Johannes Anschütz and Yichao Tian for helpful conversations regarding this work.
\section{Coherent sheaves of $\oh_X$-modules}

In this section, we recall the notion of coherent sheaves of $\oh_X$-modules and review their main properties, following \cite{BSHuber}. This will be put to use to obtain a better understanding of the notions of closed immersions and finite morphisms. During the whole section, we will use the notation established during the introduction.

\begin{defn}
The category of sheaves of $\oh_X$-modules, denoted by $\text{Mod}_{\oh_X}$, is the category whose objects are sheaves $\mathcal{F}$ on $X$ with values in the category of complete topological groups together with a continuous $\oh_X$-action such that, for all open subsets $U \subseteq X$, $\mathcal{F}(U)$ becomes an $\oh_X(U)$-module; and whose morphisms are maps of sheaves respecting the $\oh_X$-action.
\end{defn}

Assume now we are given an ring of integral elements $A^+ \subseteq A$, and a finite $A$-module $M$. Then one can form a presheaf $\widetilde{M}$ of complete topological groups on $X=\spa(A, A^+)$ such that $\widetilde{M}(U)=M\otimes_A \oh_X(U)$ for all rational subsets $U \subseteq X$, where the right hand side is endowed with the natural topology, and the restriction map $\widetilde{M}(U) \rightarrow \widetilde{M}(V)$ is the $M$-tensoring over $A$ of the rational localisation $\oh_X(U) \rightarrow \oh_X(V)$, for any pair of rational subsets $V \subseteq U$. This presheaf carries an obvious action of $\oh_X$ and it turns out to be an acyclic sheaf of abelian groups (cf. \cite[Theorem 2.5]{FRHuber}). We are naturally led to the definition of a coherent sheaf:

\begin{defn}
We say that a sheaf of $\oh_X$-modules $\mathcal{F}$ is coherent if there is an affinoid cover $X=\bigcup_i U_i$ such that, for all $i$, $\mathcal{F}\big{|}_{U_i} \cong \widetilde{M_i}$ for some finite $\oh_X(U_i)$-module $M_i$.
\end{defn}

Over affinoids, we do not actually encounter any new coherent sheaves: 

\begin{thm}\label{modulos coerentes sobre afinoides sao induzidos por modulos sobre as seccoes globais}\emph{\cite[Satz 3.3.12]{BSHuber}}
Let $A$ be a strongly noetherian complete Tate ring, $A^+$ a ring of integral elements of $A$ and $X=\emph{\spa}(A, A^+)$ the associated adic space. Then, the assignment $M \mapsto \widetilde{M}$ defines an equivalence between the category of finite $A$-modules and the category of coherent sheaves on $X$.
\end{thm}

Now we turn to our application of this theorem, namely to a better comprehension of closed immersions and finite morphisms, defined in \cite[Section 1.4]{EtHuber}. Let us first quickly observe that, given a coherent sheaf of $\oh_X$-algebras $\mathcal{A}$, we may construct its relative adic spectrum $\spa \mathcal{A}$ as the adic space representing the functor $\Hom(\mathcal{A}, -):\text{Ad}_X \rightarrow \text{Set}$ which maps an adic space $Y$ over $X$ to the set $\Hom_{\oh_X-\textrm{Alg}}(\mathcal{A}, f_*\oh_Y)$. Hence it becomes clear that every finite adic space over $X$ is the relative adic spectrum of a coherent sheaf of $\oh_X$-algebras, unique up to isomorphism. As a corollary, one obtains

\begin{cor}
Under the hypothesis of Theorem \ref{modulos coerentes sobre afinoides sao induzidos por modulos sobre as seccoes globais}, every adic space $Y$ finite over $X$ is of the form $\emph{\spa}(B, B^+)$ for some finite $A$-algebra $B$. In particular, closed immersions into $X$ correspond bijectively to ideals of $A$.
\end{cor}

We may also use coherent ideal sheaves to define Zariski closed subsets:

\begin{defn}
Let $X$ be a locally noetherian analytic adic space. A subset $Z \subseteq X$ is Zariski closed if it equals the support $V(\mathcal{I})=\{ x \in X: \mathcal{I}_x \neq \oh_{X,x}\}$ of some coherent sheaf of ideals $\mathcal{I}$ on $X$.
\end{defn}

Obviously, there are several coherent ideals with the same support. The usual way around this in algebraic geometry is to focus on radical ideals, but in the adic realm we do not know if reducedness is stable under rational localisations. We do have the following safeguard, whose proof we leave to the reader.

\begin{prop}\label{quando os ideais definem o mesmo fechado}
Keep the notation of Theorem \ref{modulos coerentes sobre afinoides sao induzidos por modulos sobre as seccoes globais} and assume that $A$ is a Jacobson ring. Then two ideals $I_i$, $i=1, 2$ define the same Zariski closed subset on $X$ if and only if they have the same radical.
\end{prop}

\section{Jacobson adic spaces}

The aim of this section is to define a new class of analytic adic space containing a subset of distinguished points controlling several of its properties, very much like classical points of rigid spaces.

\begin{defn}
We say that a strongly noetherian complete Tate ring $A$ is a Jacobson-Tate ring if it satisfies the following properties:
\begin{enumerate}
\item Every residue field of $A$ is a complete non-archimedean field.
\item For every topologically of finite type $A$-algebra $B$, the induced map $\spec B \rightarrow \spec A$ respects maximal ideals.
\end{enumerate}	
\end{defn}

The reason behind the naming of these rings resides in \cite[Chapitre V §3.4, Théorème 3 (ii)]{BouAC}, which is actually an equivalence. Using these rings, we may now define the class of Jacobson adic spaces.

\begin{defn}
We say that $X$ is a Jacobson adic space if it is locally of the form $\spa(A, A^+)$, where $A$ is a Jacobson-Tate ring. We define its Jacobson-Gelfand spectrum $JG(X)$ as the subset of rank $1$ points $x \in X$ for which there is an open affinoid $\spa(A, A^+) \subseteq X$ with $A$ Jacobson-Tate such that $\text{supp}\, x \subseteq A$ is a maximal ideal.  
\end{defn}

\begin{prop}\label{propriedades de aneis de Jacobson-Tate}
Let $A$ be a Jacobson-Tate ring. Then the following hold:
\begin{enumerate}
\item Every topologically of finite type $A$-algebra $B$ is a Jacobson-Tate ring.
\item Let $B$ be a rational localisation of $A$. Then every maximal ideal $\nf \subseteq B $ is the extension of a unique maximal ideal $\mf \subseteq A$ and the natural map $\widehat{A}_{\mf} \rightarrow \widehat{B}_{\mf B}$ is an isomorphism. In particular, if $x \in X=\emph{\spa}(A, A^+)$ is the unique rank $1$ point supported at $\mf$, then the canonical map $\widehat{A}_{\mf} \rightarrow \widehat{\oh}_{X,x}$ is an isomorphism.
\item For every ring of integral elements $A^+ \subseteq A$, the support map $\emph{JG}(\emph{\spa}(A, A^+)) \rightarrow \emph{\spec} A$ is an embedding onto $\emph{Max}\,A $ and the inclusion $\emph{JG}(\emph{\spa}(A, A^+)) \rightarrow \emph{\spa}(A, A^+)$ is dense. In particular, $A$ is a Jacobson ring.
\end{enumerate}
\end{prop}

\begin{proof}
For the first claim, let $\nf \subseteq B$ be a maximal ideal. This pulls back to a maximal ideal $\mf$ of $A$, whose quotient is a complete non-archimedean field. Thus, we may assume that $A$ is a complete non-archimedean field and $B/\nf$ must be a field, by the Noether normalisation lemma for affinoid algebras. One checks similarly the second condition for $B$ being a Jacobson-Tate ring.
	
Since rational localisations are topologically of finite type homomorphisms, it is clear that $\nf$ lies over a maximal ideal $\mf \subseteq A$. But if $A=K$ is a complete non-archimedean field, any nonzero rational localisation is isomorphic to $K$. This shows that $\nf=\mf B$. As for the isomorphy claim, one shows that $A/\mf^n \rightarrow B/\mf^nB$ is an isomorphism by induction, using flatness of $ A \rightarrow B$ (cf. \cite[Proposition 3.3.16]{BSHuber}).

The last claim follows immediately from the first and the second claims. As for $A$ being Jacobson, it suffices to show density of the maximal spectrum inside the spectrum of every homomorphic image of $A$, but the adic topology is finer than the Zariski topology.
\end{proof}

Before developing some more theory around Jacobson adic spaces, let us give a sufficient criterion for an analytic adic space to be Jacobson. We need to introduce some material from \cite{Abb}.

\begin{defn}
A complete noetherian adic ring $R$ is called a valuative order if it is local, integral and one-dimensional.  
\end{defn}

The importance of these rings is revealed by the following result.

\begin{prop}\label{caracterizacao de ordens 1-valuativas}\cite[Proposition 1.11.8]{Abb}
Let $R$ be a complete noetherian adic ring, $X=\emph{\spec} R-V(R^{\circ \circ})$ and $\pf \subseteq R$ a prime ideal. The following are equivalent:
\begin{enumerate}
\item $R/\pf$ is a valuative order.
\item $\pf \nsubseteq R^{\circ \circ}$ and $\dim(R/\pf)=1$.
\item $\pf$ determines a closed point of $X$.
\end{enumerate}
\end{prop}

\begin{prop}\label{classe chave de exemplos de aneis de Jacobson-Tate}
Assume $A$ has a noetherian ring of definition $A_0 \subseteq A$ whose reduction $A_0/A^{\circ \circ}$ is a Jacobson ring. Then, $A$ is a Jacobson-Tate ring.
\end{prop}

\begin{proof}
The fact that every homomorphic field of $A$ is non-archimedean can be found in \cite[Proposition 2.2.10]{BSHuber} and it only uses noetherianness of $A_0$. For the second property, we let $B$ be a topologically of finite type $A$-algebra and choose a noetherian ring of definition $B_0$ of $B$ which is topologically finitely generated over $A_0$. Appealing to Proposition \ref{caracterizacao de ordens 1-valuativas}, we simply have to show that the map $\spec B_0 \rightarrow \spec A_0$ respect prime ideals whose residue ring is a valuative order. Without loss of generality, assume that $B_0$ is a valuative order and $A_0$ injects in the former ring. Now \cite[Proposition 1.11.2]{Abb} yields finiteness of $A_0 \rightarrow B_0$, so $A_0$ must be local, integral and $1$-dimensional.
\end{proof}   

To conclude this section, we will now transfer some ring-theoretic quantities and absolute properties to the world of Jacobson adic spaces.

\begin{defn}
Let $X$ be a Jacobson adic space. Its Krull dimension $\text{dim.Kr}\,X$ is the quantity $\sup_{x \in \text{JG}(X)} \text{dim}\,\widehat{\oh}_{X,x}$.
\end{defn}

\begin{prop}
Let $(A, A^+)$ be a Jacobson-Tate pair and let $X$ be the corresponding adic space. Then, the equality $\emph{dim.Kr} X=\emph{dim} A$ holds.
\end{prop}

\begin{proof}
Use Proposition \ref{propriedades de aneis de Jacobson-Tate}, part $2$.
\end{proof}

\begin{rem}
When $A^+ \neq A^{\circ}$, the Krull dimension may in general differ from the spectral dimension of $\spa(A, A^+)$ (e.g., take $A$ to be a complete non-archimedean field). If $A^+=A^{\circ}$, equality holds for affinoid algebras by \cite[Lemma 1.8.6 (ii) and Corollary 1.8.8]{EtHuber} and for pseudoaffinoid algebras, as we will see later on, but we do not know whether this is true in general.
\end{rem}

\begin{defn}
Let $\mathcal{I}$ be a coherent sheaf of ideals on a Jacobson adic space $X$. We say that $\mathcal{I}$ has height $d$ at $x\in\text{JG}(X)$ if $\text{ht}\,\widehat{\mathcal{I}}_x=d$, where $\widehat{\mathcal{I}}_x$ is the maximal-adic completion of $\mathcal{I}_x \subseteq \oh_{X,x}$. The height of $\mathcal{I}$ is the quantity $\text{ht}\,\mathcal{I}=\sup_{x \in \text{JG}(X)} \text{ht}\,\widehat{\mathcal{I}}_x \in \ZZ_{\geq 0} \cup \{-\infty\}$. 
\end{defn}

\begin{prop}
Let $(A, A^+)$ be a Jacobson-Tate pair and $X=\emph{\spa}(A, A^+)$ be the corresponding adic space. Given an ideal $I \subseteq A$, the equality $\emph{ht}\,I =\emph{ht}\,\widetilde{I}$ holds.
\end{prop}

\begin{proof}
First we observe that the height of $I$ can be computed by taking the supremum of $\text{ht}_{\widehat{A}_{\mf}} \widehat{I}_{\mf}$ over all maximal ideals $\mf \subseteq A$. Then use the isomorphism $\widehat{A}_{\mf} \cong \widehat{\oh}_{X,x}$, where $x \in \text{JG}(X)$ corresponds to $\mf$, and notice that it identifies $\widehat{I}_{\mf}$ with $\widehat{\mathcal{I}}_x$, where $\mathcal{I}=\widetilde{I}$.
\end{proof}

Appealing to Proposition \ref{quando os ideais definem o mesmo fechado}, we may now extend these notions to Zariski closed subsets.

\begin{defn}
Let $X$ be a Jacobson adic space and $Z \subseteq X$ a Zariski closed subset defined by a sheaf of coherent ideals $\mathcal{I}$. We say that $Z$ has codimension $d$ at $z \in JG(Z)$ if $\mathcal{I}$ has height $d$ at $z$ and we define its codimension as $\text{codim}_X Z=\text{ht}\, \mathcal{I}$.
\end{defn}

Next we show how to extend an absolute property of rings to the class of Jacobson adic spaces.

\begin{defn}
Let $P$ be a property of rings. We say that a Jacobson adic space $X$ has property $P$ if, for all $x \in \text{JG}(X)$, $\widehat{\oh}_{X,x}$ satisfies property $P$.
\end{defn}

A priori, this may not be such a useful notion, for we lack an understanding of whether affinoid sections satisfy $P$ or not. In our case, this will not be a concern.

\begin{prop}
Let $(A, A^+)$ be a Jacobson-Tate pair, $X$ denote the corresponding adic space and $P\in \{\emph{reduced, normal, regular}\}$ be a property of rings. If $X$ satisfies $P$, then so does $A$. If $A$ is an excellent ring, the converse holds.
\end{prop}

\begin{proof}
This follows from a direct application of \cite[Proposition 4.3.8]{Liu}, \cite[Corollaire 6.5.4]{EGAIV} and \cite[Corollaire 6.5.2]{EGAIV}, respectively.
\end{proof}

\begin{rem}
The statement holds more generally if one assumes that the formal fibres of $A$ are regular and $P$ is a local property for regular maps. Other examples of such properties can be found in \cite{EGAIV}.
\end{rem}

\section{Pseudorigid spaces}

In this section, we introduce the notion of pseudorigid spaces, which first appeared in \cite{JN17} under the same name, but was also developed independently by the author during this work.
  
\begin{defn}\label{definicao das algebras pseudo-afinoides}
A pseudoaffinoid\footnote{In \cite{JN17}, these are called Tate formally of finite type $\oh_K$-algebras, but we prefer to avoid this terminology.} $\oh_K$-algebra is a complete Tate $\oh_K$-algebra $A$ having a noetherian ring of definition $A_0$, which is formally of finite type over $\oh_K$. A homomorphism of pseudoaffinoid $\oh_K$-algebras $A$ and $B$ is a continuous $\oh_K$-homomorphism $\phi:A \rightarrow B$.
\end{defn}

We will need the following results on homomorphisms of pseudoaffinoid $\oh_K$-algebras.

\begin{lem}\label{mapas de algebras pseudoafinoides sao topologicamente de tipo finito}
Let $A$ be a pseudoaffinoid $\oh_K$-algebra.
\begin{enumerate}
\item Every topologically of finite type $A$-algebra $B$ is a pseudoaffinoid $\oh_K$-algebra and if $\phi:(A, A^{\circ}) \rightarrow (B, B^+)$ denotes the corresponding topologically of finite type morphism, then $B^+=B^{\circ}$.
\item \cite[Corollary A.14]{JN16} Every homomorphism of pseudoaffinoid $\oh_K$-algebras $A \rightarrow B$ is topologically of finite type.
\end{enumerate}
\end{lem}

\begin{proof}
It is clear that $B$ is a pseudoaffinoid $\oh_K$-algebra. As for $B^+ =B^{\circ}$, this is due to $B^+$ being the integral closure of a noetherian ring of definition $B_0 \subseteq B$, and thus equal to $B^{\circ}$ (cf. \cite[Lemma A.2]{JN16}).
\end{proof}

Hence, every rational subset of $\spa(A, A^{\circ})$, where $A$ is a pseudoaffinoid $\oh_K$-algebra, is of the form $\spa(B, B^{\circ})$, where $B$ is a pseudoaffinoid $\oh_K$-algebra. We now define the following subcategory of analytic adic spaces.

\begin{defn}\label{definicao dos espacos pseudo-rigidos}
A pseudorigid space over $\oh_K$ is a $\spa(\oh_K)$-analytic adic space $X$, which is locally $\spa(\oh_K)$-isomorphic to $\spa A$, where $A$ is a pseudoaffinoid $\oh_K$-algebra. The category of pseudorigid spaces over $\oh_K$ is the full subcategory of $\spa(\oh_K)$-adic spaces whose objects are pseudorigid spaces.
\end{defn}

Now Lemma \ref{mapas de algebras pseudoafinoides sao topologicamente de tipo finito} tells us that this category is stable under open immersions and the functor $A \mapsto \spa(A)$ from the category of pseudoaffinoid algebras to the category of pseudorigid spaces, is fully faithful.

\begin{exmp}\label{exemplos canonicos de algebras pseudoafinoides}
Let $\Xf$ be a formal scheme, formally of finite type over $\oh_K$, which we regard here as an adic space in the sense of \cite{FRHuber} or \cite{EtHuber}. Then the set of analytic points $X=\Xf_a$ is a pseudorigid space over $\oh_K$. For instance, if $\Xf=\spf \oh_K[[T]]$, then $X$ is a quasicompact analytic adic space, covered by two affinoids, namely $\spa K \langle T/p \rangle$ and $\spa \oh_K[[T]] \langle p/T \rangle [1/T]$. If $\oh_K$ is a mixed characteristic DVR, then the latter ring does not live over a field. Further discussion regarding this example can be found in \cite{BerkSch}.
\end{exmp}

Before moving on, let us recall that $\spa \oh_K$ consists of two points: the closed point $s=\spa k$ and the generic point $\eta=\spa K$. By \cite[Proposition 1.2.2]{EtHuber}, the fibre products $X_s=X\times_{\spa\oh_K} \spa k$ and $X_{\eta}=X\times_{\spa \oh_K} \spa K$ exist and we call these the \emph{special}, resp. \emph{generic} fibre of a pseudorigid space $X$. It is not very difficult to see that $X_s$ is the Zariski closed adic subspace defined by the ideal sheaf $\pi \oh_X$ and $X_{\eta}$ its open complement.

\begin{prop}
The generic fibre $X_{\eta}$ of a pseudorigid space $X$ over $\oh_K$ is a $K$-rigid space. If $X=\emph{\spa} A$ with $A$ pseudoaffinoid, then $X_s$ is a rigid space over a Laurent series field $k((T))$.
\end{prop}

\begin{proof}
We may assume that $X=\spa A$ with $A$ being a pseudoaffinoid algebra over $\oh_K$ and $\pi \in A^{\times}$. Then the canonical map $K \rightarrow A$ is topologically of finite type by Lemma \ref{mapas de algebras pseudoafinoides sao topologicamente de tipo finito}. On the other hand, if $\pi A=0$, we choose a topologically nilpotent unit $T \in A$ and look at the continuous $\oh_K$-homomorphism $k((T)) \rightarrow A$. Once again, this must be topologically of finite type, which yields the claim.
\end{proof}
The next proposition allows us to bring all the notions developed in the previous section to the pseudorigid world.

\begin{prop}\label{pseudo-rigido e Jacobson-Tate e excelente}
Pseudoaffinoid $\oh_K$-algebras are Jacobson-Tate rings and they admit excellent rings of definition. In particular, pseudorigid $\oh_K$-spaces are excellent Jacobson adic spaces.
\end{prop}

\begin{proof}
Let $R$ be a formally of finite type $\oh_K$-algebra. Its reduction $R/R^{\circ \circ}$ is a finite type $k$-algebra, so the first assertion is a consequence of \ref{classe chave de exemplos de aneis de Jacobson-Tate}. As for excellency of $R$, this is a consequence of \cite[Proposition 7]{PSVal} and \cite[Theorem 9]{ExcVal}.
\end{proof}

In order to understand pseudoaffinoid algebras better, we will need to work with some very concrete examples.

\begin{defn}
Given a positive rational number $\lambda =\frac{n}{m}$, $(n,m)=1$, we define the $\lambda$-elementary pseudoaffinoid $\oh_K$-algebra as $D_{\lambda}:=\oh_K[[T]] \langle \frac{\pi^m}{T^n}\rangle[1/T]$ and we name the topologically freely generated $D_{\lambda}$-algebras $D_{\lambda} \langle X_1, \dots, X_r \rangle$ by $\lambda$-canonical pseudoaffinoid $\oh_K$-algebras.
\end{defn}

For legibility purposes, we will often shorten $D_{\lambda} \langle X_1, \dots, X_r \rangle$ to $D_{\lambda, r}$. 

\begin{lem}\label{caracterizacao baril dentro do espaco do demonio} For any pseudoaffinoid $\oh_K$-algebra $A$, there is a sufficiently small $\lambda \in \QQ_{>0}$ such that $A$ is a topologically of finite type $D_{\lambda}$-algebra. \end{lem}

\begin{proof}
Let $s \in A$ be a topologically nilpotent unit and choose a sufficiently large $n$ such that $\pi^ns^{-1} \in A^{\circ}$. This allows us to define a continuous $\oh_K$-homomorphism $D_{\frac{1}{n}}\rightarrow A$ by mapping $T$ to $s$, and we are now done by Lemma \ref{mapas de algebras pseudoafinoides sao topologicamente de tipo finito}. \end{proof}

\begin{rem}
This implies that pseudoaffinoid algebras over an equicharacteristic DVR are just affinoid algebras over some field $k((T))$. Indeed, writing $\oh_K=k[[\pi]]$, which is always possible in this case, we obtain $D_{\lambda}=k((T)) \langle \pi, \pi^m/T^n \rangle$.
\end{rem}
\begin{prop}\label{dominios de ideais principais e dominios de factorizacao unica}
The pseudoaffinoid algebra $D_{\lambda, r}$ is a regular domain, whose maximal ideals have constant height equal to $r+1$.
\end{prop}

\begin{proof}
First, we observe that the completion $\oh_K[[T]]\langle X_1, \dots, X_r \rangle[1/T] \rightarrow D_{\lambda,r}$ is a regular map (cf. \cite[Tag 07BZ]{StacksProj}). Indeed, this follows from \cite[Scholie 7.8.3 (v)]{EGAIV}, which is applicable due to excellency of formally of finite type $\oh_K$-algebras, and stability of regular maps under localisation. Repeating the same argument, we even conclude that $\oh_K \rightarrow D_{\lambda, r}$ is a regular map, so $D_{\lambda, r}$ is a regular ring by \cite[Corollaire 6.5.2 (ii)]{EGAIV}.

Let now $\mf$ be a maximal ideal of $D_{\lambda, r}$. If $\mf$ defines a point in the generic fibre, its height may be computed after rational localisation. But the generic fibre of $\spa D_{\lambda,r}$ is simply an $r+1$-dimensional $K$-rigid semiopen annulus given by the inequalities $\lvert \pi^m \rvert \leq \lvert T \rvert <1$. On the other hand, if $\mf$ defines a point in the special fibre, then $\oh_K$-flatness of $D_{\lambda, r}$ and Krull's Hauptidealsatz tell us that the $\text{ht}\, \mf= \text{ht}\, \widetilde{\mf}+1$, where $\widetilde{\mf}$ is the reduction of $\mf$ modulo $\pi$. But $D_{\lambda, r}/\pi \cong k((T)) \langle X_1, \dots, X_r \rangle$ is $r$-equidimensional.

As for connectedness, this follows easily from $\oh_K$-flatness of $D_{\lambda,r}$ and both its generic and special fibres being connected.
\end{proof}

\begin{rem}
We should also note that, if one follows step by step the proof of \cite[Theorem 2.5.3]{JN17}, then one is capable of showing that $D_{\lambda,r}$ is a UFD whenever $\lambda \in \NN$.
\end{rem}

Since pseudoaffinoid algebras are catenary, the equidimensionality also holds for integral domains.

\begin{cor}\label{maximais tem altura constante num dominio ca dos nossos e a fibra especial e equidimensional}
Let $A$ be an integral pseudoaffinoid $\oh_K$-algebra. Then $A$ is an equidimensional ring and the irreducible components of $A/\pi A$ have the same dimension.
\end{cor}

The next corollary provides an answer to the question \cite[page 9]{JN17}.

\begin{cor}\label{dimensao topologica e de krull coincidem}
Let $X$ be a pseudorigid space over $\oh_K$. Then, we have an equality $\emph{dim.Kr}\,X=\emph{\dimn} X$.
\end{cor}

\begin{proof}
Since all generalisations are vertical, we assume without loss of generality that $X=\spa A$, where $A$ is a pseudoaffinoid domain. Due to Corollary \ref{maximais tem altura constante num dominio ca dos nossos e a fibra especial e equidimensional}, we know that $\dimn A=\dimn B$ for every rational localisation of $A$. 
	
Notice that the special and generic fibres are specialising subsets which partition $X$, so we have the identity $\dimn X=\max\{\dimn X_s, \dimn X_{\eta}\}$. By \cite[Lemma 1.8.6]{EtHuber}, we deduce that $\dimn X_s=\dimn A/\pi A$ and $\dimn X_{\eta}=\dimn A$, as long as the generic fibre is nonempty. But if this were the case, then $\pi A=0$. In any case, the desired equality $\dimn X=\dim A=\text{dim.Kr}\,X$ holds.
\end{proof}

The next result is a version of Noether's normalisation lemma, which will be extremely useful for proving the Hebbarkeitssätze in the subsequent section.

\begin{prop}\label{Noether levantada do chao} \emph{(Noether levantada do chão)}\footnote{This is a reference to the novel \textit{Levantado do Chão} (eng. \textit{Raised from the Ground}) from the Nobel Prize winning author José Saramago. This is motivated by the fact that we always imagine the special fibre depicted as the horizontal axis of the picture in \cite[Figure 1, page 23]{BerkSch}.}
Let $A$ be an $\oh_K$-flat pseudoaffinoid algebra, such that $\pi \notin A^{\times}$, and let $T$ be a topologically nilpotent unit of $A$. If $\phi: k((T)) \langle X_1, \dots, X_r \rangle \rightarrow  A/\pi A$ is a finite injection of $k((T))$-affinoid algebras, then it can be lifted to a finite injection $\Phi:D_{n,r} \rightarrow A \langle \pi/T^n \rangle$ of pseudoaffinoid algebras, for a sufficiently large natural number $n$.
\end{prop}

\begin{proof}
Choose a ring of definition $A_0 \subseteq A$. Since its reduction modulo $\pi$ is still a ring of definition, we deduce that $(A/\pi A)^{\circ}$ is integral over $A_0/(\pi A \cap A_0)$.
	
Let $f_i\in A$, $i=1, \dots, r$ be lifts of $\phi(X_i)$, respectively. The previous paragraph gives us a monic polynomial $p_i \in A_0[S]$ such that $p_i(f_i) \in \pi/T^{k_i}A_0$. By choosing a sufficiently large $n$ and replacing $A$ with the rational localisation $A \langle \pi/T^n \rangle$, we may assume that $p_i(f_i) \in A_0$. Therefore, we may define the map of pseudoaffinoid $\oh_K$-algebras $\Phi: D_n \langle X_1, \dots, X_r \rangle \rightarrow A$ by mapping $X_i$ to $f_i$ for all $i=1, \dots, r$, and we notice that this obviously lifts the map $\phi$.
	
Let $y_j \in A$, $j=1, \dots, k$ be topological generators of $A$ as a $D_n$-algebra. Since $\phi$ is a finite map of Huber pairs, we may find monic polynomials $q_j$ with coefficients in $D_N^{\circ} \langle X_1, \dots, X_r \rangle $ such that $q_j(y_j) \in \pi/T^{l_j} A_0$. Repeating the previous trick, we may also assume that $q_j(y_j) \in A^{\circ \circ}$. Now \cite[Lemma 1.4.3]{EtHuber} yields finiteness of $\Phi$.
	
To show injectivity, assume on the contrary that there is a nonzero $x \in \ker \Phi$. As $\phi$ is injective and $A$ is $\pi$-torsion free, one duduces that $x \in \pi^k D_N \langle X_1, \dots, X_r \rangle$ for all $k \geq 1$. But Krull's intersection theorem says that such an element is a zero divisor, which contradicts integrality of $D_{n,r}$.
\end{proof}

Later we will also need the following:
\begin{lem}\label{enorme extensao completa de escalares e fielmente plana}
Let $A$ be a pseudoaffinoid $\oh_K$-algebra and $K \subseteq L$ the completion of some weakly unramified algebraic extension of $K$. Then the continuous homomorphism $ A \rightarrow A_L:=A \widehat{\otimes}_{\oh_K} \oh_L$ is faithfully flat.
\end{lem}

\begin{proof}
Consider an injection of finite $A$-modules $N \rightarrow M$ endowed with the natural topology. First, we notice that the map of topological $A\otimes_{\oh_K} \oh_{L}$ modules $N \otimes_{\oh_K} \oh_L \rightarrow M\otimes_{\oh_K} \oh_L$ is a strict injection. Indeed, if $M_0$ is a $A_0$-lattice generating the topology on $M$, then the subspace topology on $N \otimes_{\oh_K} \oh_L$ is generated by the sets $s^nM_0\otimes_{\oh_K} \oh_L \cap N\otimes_{\oh_K} \oh_L$ for some topologically nilpotent unit $s$. By flatness of $\oh_K \rightarrow \oh_L$, we deduce that this intersection is just $(s^nM_0 \cap N) \otimes_{\oh_K} \oh_L$, and these sets must generate the natural topology, by the Banach open mapping theorem over $A$. Hence, if we complete these two modules, we still obtain an injection and, arguing as in \cite[Proposition 10.13]{Ati}, we conclude that this is the desired injection $N\otimes_A A_L \rightarrow M \otimes_A A_L$.
	
Now it is enough to show that the canonical map on adic spectra is surjective. Observe that we have an isomorphism of topological spaces $\spa(A \widehat{\otimes}_{\oh_K} \oh_L) \cong \varprojlim_{M} \spa(A \otimes_{\oh_K} \oh_{M})$, where $M/K$ runs over all finite subextensions of $L/K$ (cf. \cite[Remark 2.6.3]{KedLiu}). Since the fibres of $\spa(A \otimes_{\oh_K} \oh_M) \rightarrow \spa(A)$ are finite nonempty discrete topological spaces, we may apply \cite[Chapitre 1, \S9 6 Corollaire 1]{BouTG} to deduce surjectivity of the map.
\end{proof} 

\section{Hebbarkeitssätze}

In this section, we are going to prove the Riemannian Hebbarkeitssätze for pseudorigid spaces over $\oh_K$. Let us begin by stating these theorems:

\begin{thm}
Let $X$ be a normal pseudorigid space over $\oh_K$ and $Z$ be a Zariski closed subset of $X$. The following statements hold:
\begin{description}[style=unboxed, leftmargin=0cm]
\item[First Hebbarkeitssatz] If $\emph{codim}_X Z \geq 1$, then the restriction map $\oh_X^+(X) \rightarrow \oh_X^+(X\setminus Z)$ is an isomorphism of rings, i.e. every locally powerbounded function $f$ defined on the complement of $Z$ admits a unique locally powerbounded extension to $X$.
\item[Second Hebbarkeitssatz] If $\emph{codim}_X Z \geq 2$, then the restriction map $\oh_X(X) \rightarrow \oh_X(X\setminus Z)$ is an isomorphism of rings, i.e. every function $f$ defined on the complement of $Z$ admits a unique extension to $X$.
\end{description}
\end{thm}

For rigid spaces, these were proven in the seventies by Bartenwerfer (cf. \cite[Riemann I]{Bar}) and Lütkebohmert (cf. \cite[Satz 1.6]{Luet}). Hence, in the following pages, we may safely assume that $\oh_K$ is a mixed characteristic DVR, $X=\spa A$ is the adic spectrum of an $\oh_K$-flat normal pseudoaffinoid domain $A$ with a connected special fibre, and $Z$ is contained in the special fibre.\footnote{Here we are implicitly using the fact that the extension of a locally powerbounded function $f$ remains locally powerbounded. One argument for this is to note that the constructible set $\{x \in X: \lvert f(x) \rvert >1 \}$ is contained in $Z$, but also contains a nonempty open subset, by \cite[Korollar 3.5.7]{BSHuber}, which is impossible.}

The following is a generalisation of \cite[Hilfssatz 1.7]{Luet}.

\begin{lem}\label{provar cenas en passant finito} \emph{(Holomorphic Transfer)} Let $\phi:A \rightarrow B$ be a finite injection of pseudoaffinoid domains and let $X=\emph{\spa}A$, $Y=\emph{\spa}B$ and $\varphi=\emph{\spa}\phi$ be the corresponding adic spaces and morphisms. Let $T \subseteq Y$ be a Zariski closed subset such that $\emph{codim}_Y(T) \geq 2$. If the natural restriction map $\oh_X(X) \rightarrow \oh_X(X\setminus\varphi(T))$ is an isomorphism, then so is $\oh_Y(Y) \rightarrow \oh_Y(Y\setminus T)$.
\end{lem}

\begin{proof}
Injectivity of the map is immediate, due to connectivity of $Y$ and flatness of rational localisations. Now let $b_1, \dots, b_n \in B$ span a basis of $Q(B)/Q(A)$ and consider the free $A$-submodule $M=Ab_1 + \dots +Ab_n \subseteq B$. By construction, $B/M$ is an $A$-torsion module, so one can find a nonzero $g \in A$ such that $gB \subseteq M$. This yields a factorisation of the multiplication-by-$g$ map on the coherent sheaf $\phi_* \oh_Y$:
$$\phi_* \oh_Y \xrightarrow{\cdot g} \oh_X^{\oplus n} \hookrightarrow \phi_* \oh_Y $$
We now pick some function $f \in \oh_Y(Y\setminus T)$ and write $gf=\sum_{i=1}^{n}f_ib_i$, where $f_i \in \oh_X(X\setminus \phi(T))$, which holds over $Y\setminus \phi^{-1}(\phi(T))$. By assumption, every coordinate function $f_i$ extends to $X$, so $gf$ has an extension of the form $\sum_{i=1}^{n} f_ib_i$.
	
Let $J$ be an ideal whose vanishing set equals $T$. Since $\text{ht}\,J \geq 2$, we have that $J \nsubseteq \qf$ for all prime ideals $\qf \subseteq B$ of height $1$. As $B$ is a Jacobson ring, there is a maximal ideal $\pf \subseteq \nf \subseteq B$ not containing $I$. Let $y_{\nf} \in JG(Y)$ be the point corresponding to $\nf$. Then $(\sum_{i=1}^{n}f_ib_i) \oh_{Y,y_{\nf}}\subseteq g\oh_{Y,y_{\nf}}$, because $f$ is defined around $y_{\nf}$. Now faithful flatness of $B_{\nf} \rightarrow \oh_{Y,y_{\nf}}$ yields $(\sum_{i=1}^{n}f_ib_i) B_{\nf} \subseteq gB_{\nf}$ and thus $f \in B_{\nf} \subseteq B_{\qf}$. Applying the identity $B=\cap_{\text{ht}\,\qf=1} B_{\qf}$, we conclude that $f \in B$. 
\end{proof}

\begin{proof}[Proof of the second Hebbarkeitssatz]
We start by treating the case $A=D_n \langle X_1, \dots, X_m \rangle$ and $I= (\pi, X_1)$. Let $g$ be a function defined over $X\setminus V(I)$ and write it as a series $\sum_{i=0}^{+\infty} a_i X_1^i$ where $a_i \in \bigcap_{l \geq 0} K \langle T, \pi^n/T, T^l/\pi, X_2, \dots, X_m \rangle$. On the other hand, $g \in D_n \langle X_1^{\pm 1}, X_2, \dots, X_m \rangle$, so it can also be written as $\sum_{i=-\infty}^{i=+\infty} b_i X_1^i$, where $b_i \in D_n \langle X_2, \dots, X_m \rangle$. By uniqueness of the coefficients, we have $a_i=b_i$ for all $i \in \ZZ$, so $b_i =0$ when $i<0$ and $g=\sum_{i=0}^{+\infty} b_i X_1^i \in A$.
	
In the general case, we apply Noether's normalisation lemma to find a finite injection $\phi: k((T)) \langle X_1, \dots, X_r \rangle \rightarrow A/\pi A$ of $k((T))$-affinoid algebras such that $(X_1) \subseteq \phi^{-1}(I/\pi A)$. By Proposition \ref{Noether levantada do chao}, we can lift this to a finite injection $\Phi: D_{n,r}\rightarrow A \langle \pi/T^n \rangle$ such that $(\pi, X_1) \subseteq \Phi^{-1}(I)$. Since $A/\pi A$ is connected by hypothesis, there is a unique connected component $B$ of $A \langle \pi/T^n \rangle$ with nonempty special fibre. This yields a finite injection $D_{n,r} \rightarrow B$ of normal pseudoaffinoid domains, to which we may apply the Holomorphic Transfer Lemma.
\end{proof}

In order to show the first Hebbarkeitssatz, we want to find a powerbounded analogue of the Holomorphic Transfer Lemma. This requires recalling some properties of the trace map $tr_{B/A}$ of a finite flat homomorphism of noetherian rings. If $A$ is a normal domain, and $A \rightarrow B$ is generically étale, we can define the Dedekind different $\mathcal{D}_{B/A}$ in the usual way (cf. \cite[Tag 0BW0]{StacksProj}). This coincides with a more general definition of the different (cf. \cite[Tag 0BW4]{StacksProj} and \cite[Tag 0BW5]{StacksProj}) and its importance lies in the fact that its vanishing set coincides with the ramification locus of $B$ over $A$, by \cite[Tag 0BTC]{StacksProj}.

\begin{lem}\label{provar cenas en passant finito 1} \emph{(Powerbounded Transfer)} 
Let $\phi: A \rightarrow B$ be a finite injection of normal pseudoaffinoid domains and let $X=\emph{\spa}A$, $Y=\emph{\spa}B$ and $\varphi=\emph{\spa}(\phi)$ be the corresponding adic spaces and morphism. Let $T \subseteq Y$ be a Zariski closed subset of positive codimension. If the restriction map $A^{\circ} \rightarrow \oh_X^+(X\setminus \varphi(T))$ is an isomorphism, then every locally powerbounded function $f \in \oh_Y^+ (Y\setminus T)$ extends uniquely to the intersection of the Zariski unramified locus of $\phi$ and the preimage of the Zariski locally free locus of $\phi$. 
\end{lem}
\begin{proof}
Without loss of generality, we may assume that $\phi$ is generically étale. Let $g \neq 0 \in A$ be an element such that $B[g^{-1}]$ is a finite free $A[g^{-1}]$-module with basis $b_{1}, \dots , b_{n} \in B^+$. Over the fraction field $K=Q(A)$, this gives us a dual basis $b_{1}^{\vee}, \dots, b_{n}^{\vee} \in L=K\otimes_A B$. 
	
Consider now the trace elements $f_k=\tr(fb_{k})$ which are defined over the open set $X\setminus (\phi(T) \cup Z))$, where $Z$ is the complement of the locally free locus. Since $A$ and $B$ are normal, $Z$ has codimension at least $2$, so $f_k$ extends to the complement of $\phi(T)$. Moreover, for generically étale extensions the trace is given by a sum of Galois conjugates, so it preserves integrality, thus the functions $f_k$ are locally powerbounded. By assumption, we get $f_k \in A^{\circ}$.
	
Note that for all rational subset $V \subseteq D_X(g)$, we have an identity $f=\sum_{k=1}^n f_{k}b_{k}^{\vee}$ inside the ring $\oh_Y(\phi^{-1}(V))\otimes_A K$, by definition of the dual basis. Let $d \in B$ be an element in the different $\mathcal{D}_{B[g^{-1}]/A[g^{-1}]}$, which implies $b_{i}^{\vee} \in B[d^{-1}]$ for all $i=1, \dots, n$. Therefore, the identity above yields an extension of $f$ to the Zariski open subset where $A[g^{-1}] \rightarrow B[g^{-1}]$ is unramified.
\end{proof}

\begin{proof}[Proof of the first Hebbarkeitssatz] The proof is going to be divided into several steps.
\begin{description}[style=unboxed,leftmargin=0cm]
\item[Step 1]
Assume $A=D_n \langle X_1, \dots, X_r\rangle$. Let $f \in \oh_X^+(X\setminus Z)$ be a locally powerbounded function on the generic fibre. The latter space is a union of increasingly large semiopen rigid annuli, for which we have a canonical choice of norm, as in \cite[9.7.1, page 400]{BGR}. This norm is power multiplicative, which means that powerbounded elements are those with norm $\leq 1$. Thus $f \in \bigcap_{n<k \in \NN} \oh_K \langle T, \pi/T^n, T^k/\pi, X_1, \dots, X_r \rangle$, and, if one looks at the coefficients of its representation as a Laurent series in $T$, one concludes that $f \in D_{n,r}^{\circ}$.

\item[Step 2] Until Step 6, assume that the residue field $k$ is perfect. Here, we fix a topologically nilpotent unit $T$ of $A$ and assume that $A/\pi A$ is integral and geometrically regular over $k((T))$. Then \cite[Satz 4.1.12]{Ber} provides us a finite injection $\phi: k((S)) \langle X_1, \dots, X_r \rangle \rightarrow A$, which is generically étale. We can lift this to a finite injection $\Phi: D_{n,r}  \rightarrow A \langle \pi/T^n\rangle$, and we note that its ramification locus does not contain the special fibre. Replacing $A \langle \pi/T^n \rangle$ by the adequate connected component as in the proof of the second Hebbarkeitssatz, Lemma \ref{provar cenas en passant finito 1} shows now that $f$ is defined everywhere up to a Zariski closed subset of codimension at least $2$. The result follows now by the second Hebbarkeitssatz.

\item[Step 3] Assume that $A/\pi A$ is a regular domain. We note that $A/\pi A$ is geometrically reduced over $k((T))$ if and only if $Q(A/\pi A)\otimes_{k((T))} k((T^{1/p}))=Q(A/\pi A)[X]/(X^p-T)$ is reduced, i.e, if and only if $ T\notin Q(A/\pi A)^p$. 

At this point, we need to know that for an integral $k((T))$-affinoid algebra $B$, the field $Q(B)^{p^{\infty}}$ is a finite extension of $k$. We will give a brief sketch of the argument. First, one shows that if $L/K$ is a finite field extension, then so is $L^{p^{\infty}}/K^{p^{\infty}}$. By the Noether normalisation lemma, one is then reduced to  the case where $B=k((T)) \langle X_1, \dots, X_r \rangle$. Using the unique factorisation property of $B$ and looking at the Gauß norm, we deduce that any nonzero $b \in Q(B)^{p^{\infty}}$ must lie in $(B^{\circ})^{\times}$, whence its reduction $\overline{b} \in \overline{B}=B^{\circ}/B^{\circ\circ}=k[X_1, \dots, X_r]$ is also a $p^{\infty}$-power and thus equal to some constant $c \in k$. Running the same argument with $b-c$ instead, we conclude that it must be zero.

Therefore, there is a maximal $n \in \NN$ such that $T \in Q(A/\pi A)^{p^n}$ and, using regularity of $A/\pi A$, we find an element $S \in A$ such that $S^{p^n}\equiv T$ mod $\pi$. After shrinking $X$, we may assume that $S$ is a topologically nilpotent unit of $A$, thus turning $A/\pi A$ into a geometrically reduced $k((S))$-affinoid algebra and we return to Step 2.

\item[Step 4] Assume that $A/\pi A$ is generically reduced. By excellency of affinoid algebras, the singular locus of $\spec A/\pi A$ has codimension $2$ in $X$, so this case follows from the second Hebbarkeitssatz and the Step 3.

\item[Step 5] Here we only assume that $k$ is perfect. We claim that there is a finite extension $L/K$ and a finite injection $A \rightarrow B$, where $B$ is a normal pseudoaffinoid $\oh_L$-domain with generically reduced special fibre. This is just some variant of the Reduced Fibre Theorem as in \cite[Tag 09IL]{StacksProj} and one may basically run the same proof. 

Indeed, let $\pf_i \subseteq A$, $i=1, \dots, n$ be the minimal primes of the special fibre. The statement on $p^{\infty}$-powers from Step 3 allows us to apply \cite[Tag09F9]{StacksProj} to the extensions $\oh_K \rightarrow A_{\pf_i}$ and obtain a common solution $L$ in the sense of \cite[Tag 09EN]{StacksProj}. Define now $B$ as a connected component of $(A \otimes_{\oh_K} \oh_L)^{\text{norm}}$, on which $A$ injects. Then $A \rightarrow B$ is a finite injection and $B$ is a pseudoaffinoid domain over $\oh_L$. We now observe that the minimal primes $\qf_j$ of $\spec B/\pi_L B$ lie over the $\pf_i$, due to Corollary \ref{maximais tem altura constante num dominio ca dos nossos e a fibra especial e equidimensional} and dimension reasons. By definition of a solution, $\oh_L \rightarrow B_{\qf_j}$ is formally smooth, which implies that the special fibre of $B$ is generically reduced.

Let now $f$ be a locally powerbounded function on $X$. Then Step 4 yields $f \in B$ and, if one looks at the integral closure $C$ of $B$ in the Galois hull of $Q(B)/Q(A)$, one deduces that $f$ is invariant by the Galois action (since this is true in the generic fibre), so $f \in A$.

\item[Step 6] Here, we drop the assumption that $k$ is perfect. We assume without loss of generality that we have a finite injection $D_n(\oh_K) \langle X_1, \dots, X_r \rangle \rightarrow A$. Let $\oh_K \rightarrow \oh_L$ be a local homomorphism of complete DVR's which is weakly unramified and lifts the residue field extension $k \rightarrow k^{\text{perf}}$. There is always one such choice, by the theory of Cohen rings. We now look at the following pushout diagram in the category of rings:
$$
\begin{tikzcd}
D_{n,r}(\oh_K) \arrow[hookrightarrow]{r} \arrow{d} & A \arrow{d}\\
D_{n,r}(\oh_L) \arrow[hookrightarrow]{r}  & A_L
\end{tikzcd}
$$
By Lemma \ref{enorme extensao completa de escalares e fielmente plana}, the vertical arrows are faithfully flat, so the bottom arrow is also injective.
	
Let $B$ be the normalisation of $A_L$ in the product of fields $Q(A) \otimes_{Q(D_{n,r}(\oh_K))} Q(D_{n,r}(\oh_L))$. By excellency, $B$ is finite over $A_L$ and the composition $A \rightarrow A_L \rightarrow B$ is injective. We now fix a basis $a_1, \dots, a_k$ of $Q(A)/Q(D_{n,r}(\oh_K))$. Arguing as in the proof of the Holomorphic Transfer Lemma, we find nonzero elements $g \in D_{n,r}(\oh_K)$ and $h \in D_{n,r}(\oh_L) $, the latter of which not divisible by $\pi$, such that $gA \subseteq D_{n,r}(\oh_K)a_1 \oplus\dots \oplus D_{n,r}(\oh_K)a_k$, $h \notin \pi D_{n,r}(\oh_L) $ and $gh B \subseteq D_{n,r}(\oh_L)a_1 +\dots +D_{n,r}(\oh_L)a_k$. 
	
Let $f$ be a locally powerbounded function on $X\setminus Z$. We have an identity $gf=\sum f_i a_i$, where the $f_i$ are functions on $\spa D_{n,r}(\oh_K)\setminus V(\pi)$. On the other hand, $f$ extends to $\spa B$ by Step 5, so we have another identity $hgf=\sum \tilde{f_i}a_i$, where the $\tilde{f_i}$ belong to $D_{n,r}(\oh_L)$. By uniqueness of coefficients, one gets $\tilde{f_i}=hf_i$ on $\spa D_{n,r}(\oh_L)\setminus V(\pi)$, so $f_i \in D_{n,r}(\oh_L)$, by the second Hebbarkeitssatz. This shows that the $f_i$ are locally bounded, so by Step 1, we also deduce that $f_i \in D_{n,r}(\oh_K)$. In conclusion, $f \in Q(A) \cap B$. 

Looking at the characteristic polynomial $p$ of $f$ over $Q(D_{n,r}(\oh_K))$, we see that it has coefficients in $D_{n,r}(\oh_L)$. But faithful flatness of $D_{n,r}(\oh_K) \rightarrow D_{n,r}(\oh_L)$ tells us that $Q(D_{n,r}(\oh_K)) \cap D_{n,r}(\oh_L)=D_{n,r}(\oh_K)$ so $f \in A$.\end{description}\end{proof}

\begin{rem}
It is also possible to prove the first Hebbarkeitssatz by adapting the formal methods of \cite[Section 7.3]{deJong} to pseudoaffinoid algebras and use our version of Noether's normalisation.
\end{rem}
\section{Applications to formal schemes}

In this section, we will give some applications of the previous results to formal schemes. These will be recurrently and without further distinction regarded as adic spaces, following \cite{FRHuber}. 

Recall that, according to \cite{TameFundGrp}, a formal scheme is reduced (resp. normal, resp. regular) if their local rings are. If $\Xf=\spf A$ with $A$ an excellent ring, then $\Xf$ is reduced (resp. normal, resp. regular) if and only if $A$ is. 

\begin{lem}
Let $\Xf$ be a normal excellent formal scheme, which is locally monogeneous. Then, the natural map $\oh_{\Xf}(\Xf) \rightarrow \oh_{\Xf_a}^+(\Xf_a)$ is an isomorphism of topological rings.
\end{lem}

\begin{proof}
We may assume that $\Xf=\spf A$, where $A$ is a normal domain with ideal of definition $sA$. Then $(\spf A)_a=\spa(A[1/s], A)$, and the map is clearly an isomorphism.
\end{proof}

Our first goal is to extend this result to all normal excellent formal schemes, which are nowhere discrete. For that, we will need to use normalised admissible blow-ups and Zariski's Main Theorem.

\begin{prop}\label{estender para pontos nao analiticos}
Let $\Xf$ be a normal excellent formal scheme, which is nowhere discrete. Then, the restriction map $\oh_{\Xf}(\Xf) \rightarrow \oh_{\Xf_a}^+(\Xf_a)$ is an isomorphism of topological rings.
\end{prop}

\begin{proof}
Without loss of generality, assume that $\Xf=\spf A$, where $A$ is an excellent normal domain. Let $I=(f_1, \dots, f_n)$ be an ideal of definition of $A$ and consider the blow-up $Y$ of $X=\spec A$ centred at $I$. This can be covered by the affine open subsets $Y_i=\spec A[f_j/f_i]/(f_i-\text{torsion})$. Consider furthermore its normalisation $Z=\tilde{Y}$, which is covered by the affine opens $Z_i=\spec B_i$, where $B_i$ is the normalisation of $A[f_j/f_i]/(f_i-\text{torsion})$ in $Q(A)$. In virtue of $A$ being excellent, $Z \rightarrow Y$ is a finite morphism. As $Z \rightarrow X$ is birational, Zariski's Main Theorem yields an isomorphism $\oh_X(X) \rightarrow \oh_Z(Z)$.
	
In the following, we denote by $\Yf$, $\Yf_i$, $\Zf$ and ${\Zf}_i$ the $I$-adic completions of $Y$, $Y_i$, $Z$ and $Z_i$, respectively. By excellency, $\Zf$ is still normal. We now look at the following diagram:
$$
\begin{tikzcd}
\oh_{\Xf}(\Xf) \arrow{d}\arrow{r}& \oh_{\Zf}(\Zf) \arrow{d}\\
\oh_{\Xf_a}^+(\Xf_a) \arrow{r}& \oh_{\Zf_a}^+(\Zf_a)
\end{tikzcd} 
$$
Let us now compute the ring on the upper right corner. The theorem on formal functions gives $\varprojlim_n H^0(Z, \oh_Z/I^n\oh_Z) \cong \varprojlim_ nH^0(Z, \oh_Z)/I^nH^0(Z, \oh_Z)$. Since $\oh_Z(Z)=A$ is already $I$-adically complete, we get $\oh_{\Zf}(\Zf)=A=\oh_{\Xf}(\Xf)$. By the previous lemma, the right arrow is also an isomorphism. Finally, admissible blow-ups induce an isomorphism of the corresponding analytic subsets (cf. \cite[(1.1.12)]{EtHuber}), so $\Yf_a \cong \Xf_a$ is already normal and thus $\Zf_a \cong \Xf_a$, which gives isomorphy of the bottom map.
	
As for the topological claim, we need to verify that the topology induced on $A$ by the ideals $H^0(\Yf, I^m\oh_{\Yf})=\varprojlim_n H^0(Y, I^m\oh_Y)/I^nH^0(Y, I^m\oh_Y)$, $m \in \NN$, is just the usual $I$-adic topology. Indeed, $I^m\oh_Y=\oh_Y(m)$ is induced by the graded module $\bigoplus_{k \geq m} I^m$, so it follows that $H^0(Y, \oh_Y(m))=I^m$ for all sufficiently large $m$, arguing as in \cite[Exercise 5.9(b)]{Harts} (which works more generally for Nagata rings, cf. \cite[Remark 5.19.2]{Harts}). 
\end{proof}

As a corollary of the previous proposition and the first Hebbarkeitssatz, we immediately retrieve a theorem of de Jong:

\begin{cor}\label{resultado do de Jong}\emph{\cite[Theorem 7.4.1]{deJong}}
Let $\Xf$ be an $\oh_K$-flat normal formal scheme, locally formally of finite type over $\oh_K$ and let $\Xf_{\eta}$ be its generic fibre. Then, the restriction map $\oh_{\Xf}(\Xf) \rightarrow \oh_{\Xf}^+(\Xf_{\eta})$ is an isomorphism of rings.
\end{cor}
\begin{proof}
Given a function $f \in \oh_{\Xf}^+(\Xf_{\eta})$, we first extend it to $\Xf_a$ using the first Hebbarkeitssatz and then to $\Xf$ using Proposition \ref{estender para pontos nao analiticos}.
\end{proof}

Before deriving another corollary of Proposition \ref{estender para pontos nao analiticos}, we need to discuss the specialisation mapping. 

\begin{defn}
We define the specialisation mapping as the continuous function $\text{sp}: \Xf \rightarrow \Xf_{\text{red}}$ which sends $x$ to the unique generic point of $\overline{\{x\}} \cap \Xf_{\text{red}}$.
\end{defn}

If $\Xf=\spa A$, the trivial valuation $y$ given by $\lvert a(y)\rvert=1$ if and only if $\lvert a(x) \rvert =1$ is the unique generic point among all trivial specialisations of $x$, so the specialisation mapping is well defined and it is functorial in $\Xf$. We should also remark that, when $x \in JG((\spa A)_a)$, its support on $A$ defines a valuative order by \cite[Proposition 1.11.2]{Abb}. The maximal ideal of this order defines a closed point in $A_{\text{red}}$ and, by the concrete description given above, coincides with $\text{sp}(x)$. Hence this generalises the definition of the specialisation mapping given in \cite[7.1.5]{BGR}.

\begin{cor}
Let $\mathcal{C}$ be the category of normal formal schemes, locally formally of finite type over $\oh_K$, which are nowhere discrete with adic morphisms over $\oh_K$ as morphisms of $\mathcal{C}$. Let $\mathcal{D}$ be the category of triples $(X,Y,p)$, where $X$ is a pseudorigid space over $\oh_K$, $Y$ is a $k$-scheme and $p: \lvert X\rvert \rightarrow \lvert Y\rvert$ is a map of topological spaces. If the residue field $k$ is algebraically closed, then the functor $F:\mathcal{C}\rightarrow \mathcal{D}$ given by $\Xf \mapsto (\Xf_a, \Xf_{\emph{red}}, \emph{sp}_{\Xf})$ is fully faithful. 
\end{cor}

\begin{proof}
Let $\Xf, \Yf \in \text{Obj} X$ and $f,g: \Yf \rightarrow \Xf$ be adic morphisms such that $F(f)=F(g)$. In particular, the topological maps $\lvert f \rvert$ and $\lvert g \rvert$ coincide. Hence, we may assume that $\Yf=\spf B$, $\Xf=\spf A$ and denote by $\phi_f, \phi_g: A \rightarrow B$ the homomorphisms inducing $f$ and $g$, respectively. By Proposition \ref{estender para pontos nao analiticos}, we deduce that $\phi_f=\phi_g$ and thus $f=g$.
	
Consider now a morphism of triples $(\varphi, \psi):(\Yf_a, \Yf_{\text{red}}, \text{sp}_{\Yf}) \rightarrow (\Xf_a, \Xf_{\text{red}}, \text{sp}_{\Xf})$. Let $\spf A \subseteq \Xf$ be an affine open subset and cover $\psi^{-1}(\spec A_{\text{red}})=\bigcup_{i=1}^n \spec (B_i)_{\text{red}}$, where $\spf B_i \subseteq \Yf$ are affine open subsets. From the definition of the specialisation map, it follows that $\varphi^{-1}((\spf A)_a)=\bigcup_{i=1}^n (\spf B_i)_a$. Since $F$ is faithful, we are reduced to the case where $\Yf=\spf B$ and $\Xf=\spf A$.
	
Now Proposition \ref{estender para pontos nao analiticos} gives us a continuous homomorphism $\phi: A \cong \oh_{\Xf}^+(\Xf_a) \rightarrow \oh_{\Yf}^+(\Yf_a) \cong B$, inducing $\varphi$ in the analytic loci, thus, in particular, $\varphi$ is adic (cf. \cite[Lemma 7.46 (2)]{Wed}). Now, it suffices to show that $\psi$ is equal to $\phi_{\text{red}}$. As the specialisation mapping surjects onto closed points (cf. \cite[Proposition 1.11.10]{Abb}), we get an equality of topological maps $\lvert\phi_{\text{red}}\rvert=\lvert\psi\rvert$. Noticing that $A_{\text{red}}$ and $B_{\text{red}}$ are reduced and finitely generated over an algebraically closed field $k$, we deduce the identity $\phi_{\text{red}}=\psi$. 
\end{proof}

If we apply instead Corollary \ref{resultado do de Jong}, we can obtain a similar result with $\Xf_a$ replaced by the generic fibre $\Xf_{\eta}$. Since the former is a rigid space, we may restrict the specialisation map to classical points and furthermore, we can replace the reduced scheme $\Xf_{\text{red}}$ by its perfection, as these are canonically homeomorphic.

\begin{cor}
\emph{(Scholze)} Let $\mathcal{C}$ be the category of $\oh_K$-flat normal formal schemes, locally formally of finite type over $\oh_K$. Let $\mathcal{D}$ be the category of triples $(X,Y,p)$, where $X$ is a $K$-rigid space (here regarded as in the classical sense), $Y$ is a perfect $k$-scheme and $p: \lvert X \rvert\rightarrow \lvert Y\rvert$ is a map of topological spaces. If the residue field $k$ is algebraically closed, then the functor $F:\mathcal{C}\rightarrow \mathcal{D}$ given by $\Xf \mapsto (\Xf_{\eta}, (\Xf_{\emph{red}})^{\emph{perf}}, \emph{sp}_{\Xf})$ is fully faithful. 
\end{cor}

\begin{proof}
Just as in the previous corollary, we reduce to the affine case $\Xf=\spa A$ and $\Yf=\spa B$. Then full faithfulness is a consequence of de Jong's theorem as $A \cong H^0(\Xf_{\eta}, \oh_{\Xf}) \rightarrow H^0(\Yf_{\eta}, \oh_{\Yf}) \cong B$ allows us to either recover or construct the map. In order to check the agreement over the reduced locus, we argue as in the previous corollary, noting that $\Xf_{\eta}$ is a Zariski open subset of $\Xf_a$.
\end{proof}

\end{document}